\documentclass{article}

\usepackage{graphicx}
\usepackage{mathrsfs}
\usepackage{enumerate}
\usepackage{multicol}
\usepackage{color}
\usepackage{amsmath,amssymb,amscd}
\usepackage{amsthm}
\usepackage{fullpage}
\usepackage{mathabx}
\usepackage{pdfpages}
\usepackage{hyperref}
\usepackage{cite}

\newcommand{\R}{\mathbb{R}}

\theoremstyle{plain}
\newtheorem{thm}{Theorem}[section]
\newtheorem{lemma}[thm]{Lemma}
\newtheorem{cor}[thm]{Corollary}

\newtheorem{que}[thm]{Question}
\newtheorem{claim}[thm]{Claim}

\theoremstyle{definition}

\theoremstyle{remark}
\newtheorem{rem}[thm]{Remark}

\makeatletter
\newcommand{\subjclass}[2][2020]{%
 \let\@oldtitle\@title%
 \gdef\@title{\@oldtitle\footnotetext{#1 \emph{MSC classes.} #2}}%
}
\newcommand{\keywords}[1]{%
 \let\@@oldtitle\@title%
 \gdef\@title{\@@oldtitle\footnotetext{\emph{Keywords:} #1.}}%
}
\makeatother

\title{On noncontinuous bisymmetric strictly monotone operations}

\author{Gergely Kiss}


\date{}

\begin{document}

\maketitle

\begin{abstract}
We construct bisymmetric, partially strictly increasing binary operations on real intervals which are not continuous. This shows that continuity cannot be removed from the classical characterization of bisymmetric quasi-sums due to Acz\'el and Dhombres.

Our examples are operations of the form
\[
F(x,y)=f^{-1}(\alpha f(x)+\beta f(y)),
\]
where $\alpha,\beta>0$ and $\alpha+\beta\neq1$. The generating map $f$ is chosen as a strictly increasing bijection from an interval onto a two-component subset of $\R$ separated by a gap. The range is directly invariant under the relevant weighted sums, while the gap produces a jump of $f$ and hence discontinuity of every one-variable section. As a consequence, we obtain discontinuous associative and strictly increasing operations on intervals, and we extend the construction to the multivariate case.

In the opposite direction, we prove that if a symmetric bisymmetric strictly increasing operation is reflexive at two points of an interval, then it is automatically continuous on the segment between them and coincides there with a quasi-arithmetic mean.
\end{abstract}

\section{Introduction}

The investigation of functional equations involving binary operations has
a long history, dating back to the foundational work of J.~Acz\'el and related studies. Among these, a central role is played by the \emph{bisymmetric equation}
\begin{equation}\label{eq:bisym}
    F(F(x,y),F(u,v)) = F(F(x,u),F(y,v)) \qquad (x,y,u,v\in I),
\end{equation}
where $I$ is an interval of the real line and $F:I^2\to I$ is a binary
operation. The bisymmetry equation is deeply connected to the theory of
means and, more generally, to the study of quasi-arithmetic and quasi-sum
type functions.

The goal of the present paper is to clarify the role of reflexivity in the
regularizing effect of bisymmetry. In particular, we show that without
reflexivity bisymmetry and strict monotonicity do not imply continuity,
even for operations of quasi-sum type.

\medskip
\noindent
\textbf{Historical background.}
The notion of bisymmetry emerged in Acz\'el’s classical work
\cite{Aczel1948}, where it was shown that \eqref{eq:bisym} provides the
functional core of the characterization of binary \emph{quasi-arithmetic means}.
Acz\'el’s approach represented a conceptual shift compared to earlier
multivariate mean characterizations by Kolmogoroff \cite{Kolmogoroff1930},
Nagumo \cite{Nagumo1930}, and de~Finetti \cite{deFinetti1931}, as it reduced
the problem to a single binary operation satisfying a small number of natural
axioms. In this setting, a quasi-arithmetic mean is defined by
\[
   F(x,y)=f^{-1}\!\left(\frac{f(x)+f(y)}{2}\right), \qquad x,y\in I,
\]
where $f:I\to\mathbb{R}$ is a continuous strictly monotone function.
Acz\'el proved that every continuous, symmetric, strictly monotone, and
bisymmetric reflexive operation on an interval is of this form
(Theorem~\ref{Aczel1} below).

In later developments, bisymmetry was studied not only for means but also
for more general \emph{weighted} operations of additive form, which we call
\emph{weighted quasi-sums}.  In the continuous setting, these are characterized by
Acz\'el and Dhombres in \cite{Aczel1989}; see Theorem~\ref{Aczel3} below.
Roughly speaking, any continuous, strictly monotone, bisymmetric operation
on an interval $I$ is (after a suitable change of variable) of the form
\[
  F(x,y)=f\big(Af^{-1}(x)+Bf^{-1}(y)+C\big),
\]
with $A,B,C\in\R$, $AB\neq0$, where $f:J\to I$ is a continuous strictly
monotone bijection between intervals.  The reflexive case corresponds to a
normalized choice of parameters with $A+B=1$ and $C=0$.

Beyond the foundational results of Kolmogoroff, Nagumo, de~Finetti and Acz\'el,
quasi-arithmetic and related means have been studied intensively from many points
of view, including Gauss composition and Matkowski--Sut\^o type equations,
equality and comparison problems, invariance and embeddability, and operator
versions; see, among many others,
Burai~\cite{Burai2013c},
Dar\'oczy--P\'ales~\cite{Daroczy2002},
Dar\'oczy--Maksa~\cite{Daroczy2013}, Duc et al.~\cite{Duc2020}, Fodor--Marichal \cite{fodor1997}, G{\l}azowska--Jarczyk--Jarczyk~\cite{Glazowska2020},
Jarczyk--Jarczyk~\cite{Jar2018},
Kiss--P\'ales~\cite{Kiss2018},
Maksa~\cite{Ma1999}, Maksa--Mokken--M\"unnich~\cite{MaMoMu2000},
Matkowski-P\'ales~\cite{MaPa2015},
Nagy--Szokol~\cite{Nagy2019},
P\'ales~\cite{Pales2011},
P\'ales--Pasteczka~\cite{PP2025},
Pasteczka~\cite{Pasteczka2020}.

\medskip
\noindent
\textbf{Without continuity assumption.}
In Acz\'el’s original proofs, the continuity of $F$ played a crucial role,
and it was unclear whether this assumption could be removed.
This step was achieved in \cite{BKSZ2021}, where Burai, Kiss, and
Szokol proved that continuity follows automatically from bisymmetry,
symmetry, strict monotonicity, and reflexivity.
In other words, every reflexive, partially strictly increasing, symmetric, and bisymmetric operation on a proper interval is automatically continuous
and therefore quasi-arithmetic (Theorem~\ref{T:bisymmetryimpliescontinuity}
below).  This established the remarkable \emph{regularity-improving} nature
of bisymmetry in the symmetric, reflexive case.

On the other hand, in the non-symmetric situation, our understanding is
much more limited.  For instance, even for a bisymmetric, reflexive, strictly
increasing operation $F$, the simple monotonicity-type implication
\eqref{eq:sign-implication}, which holds for weighted quasi-arithmetic means,
is not known to be valid in general.  Nevertheless,
it was shown in \cite{BKSZ2023} that if bisymmetry, strict monotonicity,
and reflexivity are combined with symmetry at a \emph{single} pair
$(x,y)$ with $x\neq y$, then global symmetry and continuity follow; see
Theorem~\ref{thm:one-symmetric-pair} below.

\medskip
\noindent
\textbf{Our goals.}
The present paper aims to clarify the role of reflexivity in the
regularizing effect of bisymmetry by studying bisymmetric and strictly
monotone binary operations \emph{without} assuming reflexivity.
Our main contributions can be summarized as follows.

\begin{itemize}
  \item We construct, for every pair $\alpha,\beta>0$ with
  $\alpha+\beta\neq 1$, a bisymmetric, partially strictly increasing (and in the
  case $\alpha=\beta$ also symmetric) binary operation
  \[
      F(x,y)=f^{-1}(\alpha f(x)+\beta f(y))
  \]
  on a suitable interval $I$, which is not continuous (here $f$ is a strictly increasing bijection between $I$ and a suitable set $H'$). More precisely,
  every section $F(x_0,\cdot)$ and $F(\cdot,y_0)$ is discontinuous for
  every $x_0,y_0\in I$.  This shows that bisymmetry and strict monotonicity do \emph{not} imply continuity in general, and that
  continuity cannot be dropped from Acz\'el’s quasi-sum characterization
  ({\bf Theorem~\ref{Aczel3}(2)}).

  \item As a special case $\alpha=\beta=1$, we obtain an associative,
  partially strictly increasing operation $F:I^2\to I$ which is
  discontinuous, giving a negative answer to the analogous question for
  Acz\'el’s theorem on associative quasi-sums ({\bf Theorem \ref{thm:assoc-counterexample}}).

  \item In a modified version of the associative construction we can
  additionally arrange \emph{one-point reflexivity}.  More precisely, after
  adjoining a neutral boundary point we obtain a bisymmetric, associative,
  symmetric, partially strictly increasing operation on $[0,\infty)$ which is
  reflexive only at $0$ and is not continuous
  ({\bf Remark~\ref{rem:onepoint}}).  Together with our two-point reflexivity
  result ({\bf Theorem~\ref{T:refl2p}}), this shows that a single reflexive
  point is too weak to enforce continuity, while reflexivity at two points,
  combined with symmetry and bisymmetry, already implies continuity and a
  quasi-arithmetic representation on the whole segment between them.

  \item We extend the construction and the structural results to the
  higher-dimensional case ({\bf Theorem \ref{thm:n-ary-counterexample}}), for operations of the form
  \[
    F(x_1,\dots,x_n)
      = f^{-1}\!\Big(\sum_{i=1}^n \alpha_i f(x_i)\Big),
  \]
  for positive coefficients $\alpha_i$ in the non-reflexive regime
  $\sum_i\alpha_i\neq1$, and we discuss the implications for $n$-ary
  bisymmetric operations.
\end{itemize}

In particular, our discontinuous examples show that, outside the reflexive
setting, bisymmetry and strict monotonicity do not force continuity, even
though every one-variable section remains strictly monotone.
At the same time, the two-point reflexivity theorem reveals a strong
local rigidity phenomenon (in the symmetric case, reflexivity at two
points forces a quasi-arithmetic structure on the interval spanned by these points).

\medskip
\noindent
The paper is organized as follows.
In Section~\ref{S:preliminary} we collect the necessary definitions and
recall several classical results, including Acz\'el’s characterizations of
quasi-arithmetic means and quasi-sums, as well as the continuity results
from \cite{BKSZ2021,BKSZ2023}.
Section~\ref{sec:H0} contains our main results on constructing discontinuous,
non-reflexive bisymmetric (and, respectively, associative) operations, and
their higher-dimensional extensions are presented in Section~\ref{sec:multi}.
Section~\ref{sec:two-point-refl} is devoted to the local reflexivity
principle at two points.
Finally, Section~\ref{sec:conclusion} collects concluding remarks and open
problems.
\section{Preliminaries}\label{S:preliminary}

Throughout the paper $I\subseteq\mathbb{R}$ denotes a proper interval, that is,
an interval with nonempty interior and at least two points.
We consider binary operations $F:I^2\to I$ satisfying certain natural
algebraic and regularity properties that we recall below.

\subsection*{Basic definitions}

\begin{enumerate}[(i)]
    \item \textbf{Reflexivity.}
    The operation $F$ is said to be \emph{reflexive} if
    \[
        F(x,x)=x \qquad (x\in I).
    \]
    Reflexivity expresses the mean-type property that $F$ reproduces its
    argument when both variables are equal.

    \item \textbf{Partial strict monotonicity.}
    $F$ is \emph{partially strictly increasing} if for every fixed
    $x_0,y_0\in I$ the functions
    \[
        x\mapsto F(x,y_0),\qquad y\mapsto F(x_0,y)
    \]
    are strictly increasing on $I$.
    The weaker notions ``partially monotone'' and ``partially increasing''
    are defined analogously.

    \item \textbf{Symmetry.}
    $F$ is \emph{symmetric} if $F(x,y)=F(y,x)$ for all $x,y\in I$.

    \item \textbf{Bisymmetry.}
    $F$ is \emph{bisymmetric} if
    \begin{equation}\label{E:bisymmetry}
        F(F(x,y),F(u,v))=F(F(x,u),F(y,v))
        \qquad (x,y,u,v\in I).
    \end{equation}

   \item \textbf{Associativity.}
    $F$ is \emph{associative} if
    \begin{equation}\label{E:associativity}
        F(F(x,y),z)=F(x,F(y,z))
        \qquad (x,y,z\in I).
    \end{equation}


    \item \textbf{Mean-type property.}
    $F$ is a \emph{mean} if
    \[
        \min\{x,y\}\le F(x,y)\le\max\{x,y\},\qquad (x,y\in I),
    \]
    and a \emph{strict mean} if the inequalities are strict whenever
    $x\ne y$.
\end{enumerate}

\subsection*{Acz\'el's characterization theorems}

The connection between bisymmetry and the theory of means originates from the following classical result of J.~Acz\'el.

\begin{thm}[Acz\'el, 1948; quasi–arithmetic means]\label{Aczel1}
A function $F:I^2\to I$ is continuous, reflexive, partially strictly
monotone, symmetric, and bisymmetric if and only if there exists a
continuous strictly monotone function $f:I\to\mathbb{R}$ such that
\begin{equation}\label{eq:QAM}
    F(x,y)=
    f^{-1}\!\left(\frac{f(x)+f(y)}{2}\right)
    \qquad (x,y\in I).
\end{equation}
In this case, $F$ is called a \emph{quasi–arithmetic mean} generated by $f$.
\end{thm}

In a more general setting, continuous bisymmetric operations are
characterized as follows; see Acz\'el and Dhombres
\cite[Chap.~11]{Aczel1989}.

\begin{thm}[Acz\'el–Dhombres; quasi–sums]\label{Aczel3}
Let $I$ be a proper interval and let
$F:I^2\to I$ be a partially strictly monotone and bisymmetric continuous mapping. Then:
\begin{enumerate}
   \item[\rm{(1.)}] $F$ is reflexive if and only if there exists a continuous,
   strictly monotone function $f\colon J\to I$ and a parameter
   $r\in\mathbb{R}\setminus\{0,1\}$ such that
   \begin{equation}\label{eqa1x3}
       F(x,y)
       =f\big(rf^{-1}(x)+(1-r)f^{-1}(y)\big),
       \qquad x,y\in I.
   \end{equation}
   In particular, $F$ is a weighted
   quasi-arithmetic mean.

   \item[\rm{(2.)}] In general, there exist constants $A,B,C\in\mathbb{R}$ with
   $AB\neq 0$, and a continuous, strictly monotone function
   $f\colon J\to I$ such that
   \begin{equation}\label{eqa1x2}
       F(x,y)
       =f\big(Af^{-1}(x)+Bf^{-1}(y)+C\big),
       \qquad x,y\in I.
   \end{equation}
\end{enumerate}
Here $J\subset\mathbb{R}$ is a proper interval.
\end{thm}

From now on, we refer to functions of the form \eqref{eqa1x2} as
\emph{weighted quasi--sums}.
The reflexive case \eqref{eqa1x3} corresponds to a normalization of the
parameters with $A+B=1$ after a suitable affine rescaling of $f$.

Another classical theorem of Acz\'el characterizes continuous associative
operations; see \cite[Theorem~1, p.~107]{Aczel1987}.

\begin{thm}[Acz\'el]\label{Aczel_assoc}
Let $I$ be a proper interval.  A function $F:I^2\to I$ is continuous,
associative, and partially strictly increasing if and only if there exist a
proper interval $J\subset\R$ and a continuous strictly monotone bijection
$f:J\to I$ such that
\[
    F(x,y)=f\bigl(f^{-1}(x)+f^{-1}(y)\bigr),
    \qquad x,y\in I.
\]
In particular, the sums $f^{-1}(x)+f^{-1}(y)$ belong to $J$ for all
$x,y\in I$.
\end{thm}

\subsection*{Regularity improvement by bisymmetry}

In Acz\'el’s original proofs, continuity was assumed from the beginning.
It was later shown in \cite{BKSZ2021} that in the symmetric, reflexive
case this assumption can be removed completely.

\begin{thm}[Burai–Kiss–Szokol, 2021]\label{T:bisymmetryimpliescontinuity}
If $F:I^2\to I$ is reflexive, partially strictly increasing, symmetric, and
bisymmetric, then $F$ is necessarily continuous and hence of the
quasi–arithmetic form \eqref{eq:QAM} for some continuous strictly monotone
$f$.
\end{thm}

Thus, in the symmetric reflexive case, bisymmetry together with strict
monotonicity automatically yields the full regularity required in
Acz\'el’s characterization.

In contrast, in the \emph{non-symmetric} case our understanding is still
far from complete. For example, we do not even know whether the
following simple implication holds for a bisymmetric, strictly increasing
operation $F$:
\begin{equation}\label{eq:sign-implication}
  x<y \ \text{ and }\ F(x,y)<F(y,x)
  \ \Longrightarrow\
  F(u,v)<F(v,u)\quad(\forall\,u<v).
\end{equation}
This implication is valid for weighted quasi--sums with positive weights,
but it is open in full generality.

Nevertheless, in the reflexive case a strong result is known when
symmetry holds at \emph{one} pair of distinct points.

\begin{thm}[Burai–Kiss–Szokol, 2023]\label{thm:one-symmetric-pair}
Let $I$ be a proper interval and let
$F\colon I^2\to I$ be a reflexive, partially strictly increasing, and
bisymmetric function.
Suppose that there exist $x,y\in I$ with $x\ne y$ such that
\[
   F(x,y)=F(y,x).
\]
Then $F$ is symmetric on $I$ and continuous, i.e., there exists a
continuous, strictly increasing function $f:I\to\mathbb{R}$ such that
\begin{equation}\label{eqa2}
    F(u,v)
    = f^{-1}\!\left(\frac{f(u)+f(v)}{2}\right),
    \qquad u,v\in I.
\end{equation}
In particular, $F$ is a quasi-arithmetic mean.
\end{thm}

This shows that in the reflexive case, a single ``symmetry point'' is
already enough to force global symmetry and regularity, while in the
non-reflexive regime (which is the main focus of the present paper) such
a conclusion is no longer available.

\section{Noncontinuous strictly increasing bisymmetric operations}\label{sec:H0}

In this section we provide explicit examples of bisymmetric and partially
strictly increasing operations that fail to be continuous.  The construction
is elementary.  We map an interval increasingly onto a two-component subset
of $\R$ separated by a gap.  The two components are chosen so that every
weighted sum falls into one component, while the gap produces a jump in the
parametrization.  If the parameters are chosen symmetrically, then the
resulting operation is symmetric as well.

\subsection{Weighted quasi-sums on general sets}

The following lemma collects basic algebraic properties of operations of
quasi--sum type on an arbitrary ordered set.

\begin{lemma}\label{lem:weighted-quasisum}
Let $X,Y\subset\R$ with $|X|\ge2$, and let $f:X\to Y$ be a strictly increasing
bijection.
Let $\alpha,\beta\in\R$ be such that
\[
   \alpha f(x)+\beta f(y)\in Y \qquad (x,y\in X),
\]
and define
\[
   F(x,y):=f^{-1}(\alpha f(x)+\beta f(y)), \qquad x,y\in X.
\]
Then:
\begin{enumerate}
  \item[\rm{(i)}] $F$ is bisymmetric for all $\alpha,\beta\in\R$;
  \item[\rm{(ii)}] if $\alpha,\beta>0$, then $F$ is partially strictly increasing;
  \item[\rm{(iii)}] $F$ is reflexive if and only if $\alpha+\beta=1$;
  \item[\rm{(iv)}] $F$ is symmetric if and only if $\alpha=\beta$;
  \item[\rm{(v)}] if $\alpha,\beta\ne 0$, then $F$ is associative if and only if $\alpha=\beta=1$.
\end{enumerate}
\end{lemma}

\begin{proof}
(i) For $x,y,u,v\in X$ we have
\begin{align*}
  f\big(F(F(x,y),F(u,v))\big)
  &= \alpha f(F(x,y))+\beta f(F(u,v))  \\
  &= \alpha(\alpha f(x)+\beta f(y))+\beta(\alpha f(u)+\beta f(v)) \\
  &= \alpha^2 f(x)+\alpha\beta f(y)+\alpha\beta f(u)+\beta^2 f(v).
\end{align*}
The same expression is obtained for $f(F(F(x,u),F(y,v)))$, hence by
injectivity of $f$,
\[
  F(F(x,y),F(u,v))=F(F(x,u),F(y,v)).
\]

(ii) If $\alpha,\beta>0$, then for fixed $y_0$ the function
$x\mapsto \alpha f(x)+\beta f(y_0)$ is strictly increasing, and so is
$x\mapsto F(x,y_0)=f^{-1}(\alpha f(x)+\beta f(y_0))$.
The argument in the second variable is analogous.

(iii) We have
\[
  F(x,x)=x
  \ \Longleftrightarrow\
  f(F(x,x))=f(x)
  \ \Longleftrightarrow\
  (\alpha+\beta)f(x)=f(x)
\]
for all $x\in X$.
If $\alpha+\beta=1$ this is automatic, hence $F$ is reflexive.
Conversely, if $F$ is reflexive, then $(\alpha+\beta-1)f(x)=0$ for all
$x\in X$.
Since $f$ is strictly increasing and $|X|\ge2$, $f$ is not constant, so
$\alpha+\beta-1=0$.

(iv) Symmetry means
\[
  \alpha f(x)+\beta f(y)=\alpha f(y)+\beta f(x)
\]
for all $x,y\in X$, i.e.\ $(\alpha-\beta)(f(x)-f(y))=0$.
If $\alpha=\beta$ the equality holds.
If $F$ is symmetric and $f$ is nonconstant, we can pick $x,y$ with
$f(x)\neq f(y)$, forcing $\alpha=\beta$.

(v) Associativity $F(F(x,y),z)=F(x,F(y,z))$ gives
\begin{align*}
  f(F(F(x,y),z))
  &= \alpha f(F(x,y))+\beta f(z)
   = \alpha(\alpha f(x)+\beta f(y))+\beta f(z) \\
  &= \alpha^2 f(x)+\alpha\beta f(y)+\beta f(z),
\\
  f(F(x,F(y,z)))
  &= \alpha f(x)+\beta f(F(y,z))
   = \alpha f(x)+\beta(\alpha f(y)+\beta f(z)) \\
  &= \alpha f(x)+\alpha\beta f(y)+\beta^2 f(z).
\end{align*}
Thus
\[
  (\alpha^2-\alpha)f(x)+(\beta-\beta^2)f(z)=0
\]
for all $x,z\in X$.
Again using that $f$ is nonconstant, we get $\alpha^2-\alpha=0$ and
$\beta-\beta^2=0$, hence $\alpha,\beta\in\{0,1\}$.
The assumption $\alpha,\beta\ne0$ leaves only $\alpha=\beta=1$.

Conversely, if $\alpha=\beta=1$, then
\[
  f(F(F(x,y),z))=f(x)+f(y)+f(z)=f(F(x,F(y,z))),
\]
so $F$ is associative.
\end{proof}

\subsection{The main binary counterexample}

We first state two elementary range constructions depending only on a
positive parameter $\lambda\neq1$.

\begin{lemma}\label{lem:expanding-range}
Assume that $\lambda>1$ and let
\[
   \gamma:=\frac{1+\lambda}{2}.
\]
Set
\[
   I:=(1,\infty),
   \qquad
   H':=(1,\gamma)\cup[\lambda,\infty),
\]
and define $f:I\to H'$ by
\[
   f(t):=
   \begin{cases}
      t, & 1<t<\gamma,\\[1mm]
      t+\lambda-\gamma, & t\ge\gamma.
   \end{cases}
\]
Then $f$ is a strictly increasing bijection,
\[
   \lim_{t\uparrow\gamma}f(t)=\gamma,
   \qquad
   f(\gamma)=\lambda,
\]
and
\[
   f^{-1}(s)=s-(\lambda-\gamma),\qquad s\ge\lambda.
\]
\end{lemma}

\begin{proof}
Since $1<\gamma<\lambda$, the two branches of $f$ have disjoint ordered
ranges $(1,\gamma)$ and $[\lambda,\infty)$, respectively, and together they
cover $H'$.  Thus $f$ is a strictly increasing bijection with the stated
jump at $\gamma$.  The formula for $f^{-1}$ on $[\lambda,\infty)$ follows
directly from the second branch of $f$.
\end{proof}

The case $0<\lambda<1$ is completely analogous, with the roles of the lower
and upper components reversed.  For completeness, we state and prove the
corresponding construction explicitly.

\begin{lemma}\label{lem:contracting-range}
Assume that $0<\lambda<1$ and let
\[
   \gamma:=\frac{1+\lambda}{2}.
\]
Set
\[
   I:=(0,\gamma),
   \qquad
   H':=(0,\lambda]\cup(\gamma,1),
\]
and define $f:I\to H'$ by
\[
   f(t):=
   \begin{cases}
      t, & 0<t\le\lambda,\\[1mm]
      t+\gamma-\lambda, & \lambda<t<\gamma.
   \end{cases}
\]
Then $f$ is a strictly increasing bijection,
\[
   f(\lambda)=\lambda,
   \qquad
   \lim_{t\downarrow\lambda}f(t)=\gamma,
\]
and
\[
   f^{-1}(s)=s,\qquad 0<s\le\lambda.
\]
\end{lemma}

\begin{proof}
Since $\lambda<\gamma<1$ and $2\gamma-\lambda=1$, the two branches of
$f$ map $(0,\lambda]$ and $(\lambda,\gamma)$ increasingly onto
$(0,\lambda]$ and $(\gamma,1)$, respectively.  Thus $f:I\to H'$ is a
strictly increasing bijection with the stated jump at $\lambda$.  The
formula for $f^{-1}$ on $(0,\lambda]$ follows from the first branch.
\end{proof}

We can now prove the main binary result.

\begin{thm}\label{thm:alphabeta-counterexample}
Let $\alpha,\beta>0$ satisfy $\alpha+\beta\neq1$.
Then there exist a proper interval $I\subset\R$, a set $H'\subset\R$, and a
strictly increasing bijection $f:I\to H'$ such that
\[
   \alpha H'+\beta H'\subset H'
\]
and the function
\[
   F(x,y):=f^{-1}\bigl(\alpha f(x)+\beta f(y)\bigr),
   \qquad (x,y)\in I^2,
\]
is partially strictly increasing, bisymmetric, and non-reflexive, but not
continuous.  More precisely, for every $x_0,y_0\in I$ the sections
$F(x_0,\cdot)$ and $F(\cdot,y_0)$ are discontinuous.
\end{thm}

\begin{proof}
Put $\lambda:=\alpha+\beta$.

Assume first that $\lambda>1$, and take $I,H',f$ and $\gamma$ from
Lemma~\ref{lem:expanding-range}.  If $u,v\in H'$, then $u,v>1$, and hence
\[
   \alpha u+\beta v>\alpha+\beta=\lambda.
\]
Thus
\[
   \alpha H'+\beta H'\subset(\lambda,\infty)\subset H',
\]
so the formula defining $F$ is well defined.

\begin{claim}\label{cl:expanding-sections}
If $\lambda>1$, every vertical and every horizontal section of $F$ is
discontinuous at $\gamma$.
\end{claim}

\begin{proof}
Fix $x_0\in I$ and put $c:=f(x_0)$.  For the vertical section
\[
   G(t):=F(x_0,t)
\]
we have, as $t\uparrow\gamma$,
\[
   \alpha c+\beta f(t)\longrightarrow
   B:=\alpha c+\beta\gamma,
\]
whereas at $t=\gamma$ the corresponding value is
\[
   A:=\alpha c+\beta\lambda.
\]
Since $c>1$ and $\gamma>1$, both $A$ and $B$ are strictly larger than
$\lambda$.  Thus both arguments of $f^{-1}$ lie in the same upper
component of $H'$.  In particular, the discontinuity does not arise from
crossing the gap under $f^{-1}$; rather, the jump of $f$ at $\gamma$ is
carried through the affine upper branch of $f^{-1}$.

Hence, by Lemma~\ref{lem:expanding-range},
\[
   \lim_{t\uparrow\gamma}G(t)
      =B-(\lambda-\gamma),
   \qquad
   G(\gamma)=A-(\lambda-\gamma).
\]
Therefore
\[
   G(\gamma)-\lim_{t\uparrow\gamma}G(t)
      =\beta(\lambda-\gamma)>0,
\]
so $G$ is discontinuous at $\gamma$.

For a horizontal section, fix $y_0\in I$.  The same calculation, now
varying the first variable, gives
\[
   F(\gamma,y_0)-\lim_{t\uparrow\gamma}F(t,y_0)
      =\alpha(\lambda-\gamma)>0.
\]
Thus every horizontal section is also discontinuous at $\gamma$.
\end{proof}

Now assume that $0<\lambda<1$, and take $I,H',f$ and $\gamma$ from
Lemma~\ref{lem:contracting-range}.  If $u,v\in H'$, then $0<u,v<1$, and
therefore
\[
   0<\alpha u+\beta v<\alpha+\beta=\lambda.
\]
Thus
\[
   \alpha H'+\beta H'\subset(0,\lambda)\subset H',
\]
so $F$ is well defined.

\begin{claim}\label{cl:contracting-sections}
If $0<\lambda<1$, every vertical and every horizontal section of $F$ is
discontinuous at $\lambda$.
\end{claim}

\begin{proof}
Fix $x_0\in I$ and put $c:=f(x_0)$.  For
\[
   G(t):=F(x_0,t),
\]
as $t\downarrow\lambda$ we have
\[
   \alpha c+\beta f(t)\longrightarrow
   \alpha c+\beta\gamma,
\]
whereas at $t=\lambda$ the corresponding value is
\[
   \alpha c+\beta\lambda.
\]
Since $c,\lambda,\gamma<1$, both numbers lie in $(0,\lambda)$, where
$f^{-1}$ is the identity.  Thus, exactly as in the expanding case, the
jump of $f$ is inherited by the section:
\[
   \lim_{t\downarrow\lambda}G(t)-G(\lambda)
      =\beta(\gamma-\lambda)>0.
\]
Hence every vertical section is discontinuous at $\lambda$.

The same argument in the first variable gives, for every $y_0\in I$,
\[
   \lim_{t\downarrow\lambda}F(t,y_0)-F(\lambda,y_0)
      =\alpha(\gamma-\lambda)>0,
\]
and therefore every horizontal section is discontinuous at $\lambda$.
\end{proof}

In either case, Lemma~\ref{lem:weighted-quasisum} shows that $F$ is
bisymmetric and partially strictly increasing.  Since
$\alpha+\beta=\lambda\neq1$, it is non-reflexive.  This completes the proof.
\end{proof}

\begin{rem}\label{rem:sections-and-variants}
The construction in Theorem~\ref{thm:alphabeta-counterexample} admits the
following observations.

\begin{enumerate}[(i)]

\item
The proof identifies a common discontinuity point for all vertical and
horizontal sections: it is $\gamma$ when $\alpha+\beta>1$ and
$\alpha+\beta$ when $0<\alpha+\beta<1$.

\item
The construction does not require symmetry.  However, if $\alpha=\beta$,
then $F$ is symmetric by Lemma~\ref{lem:weighted-quasisum}.  Thus
Theorem~\ref{thm:alphabeta-counterexample} also yields discontinuous
examples which are bisymmetric, symmetric, and partially strictly
increasing.

\item
In the more restrictive case $\alpha,\beta\ge1$, a Cantor-type variant of
the construction yields a stronger form of discontinuity: every vertical
and horizontal section can be made to have infinitely many discontinuity
points on every tail of the underlying interval.

Indeed, put $\lambda:=\alpha+\beta$ and let
\[
   \mathbb F_0:=\mathbb Q(\alpha,\beta).
\]
Since $\mathbb F_0$ is countable, Mycielski's theorem
\cite{Mycielski1964}, applied inside a Cantor set contained in $[1,2]$,
yields a perfect nowhere dense set $H_0\subset[1,2]$ whose finite subsets
are linearly independent over $\mathbb F_0$.  Define recursively
\[
   H_{i+1}:=H_i\cup(\alpha H_i+\beta H_i),
   \qquad
   H:=\bigcup_{i=0}^{\infty}H_i.
\]
Induction shows that every $h\in H$ has a finite representation
\[
   h=\sum_{\ell=1}^m c_\ell u_\ell,
   \qquad
   c_\ell\in\mathbb F_0,\quad c_\ell\ge0,\quad u_\ell\in H_0,
\]
and the linear independence of $H_0$ makes this representation unique.

The assumption $\alpha,\beta\ge1$ gives the required escape property
directly.  For $h\in H$, let
\[
   \ell(h):=\min\{i:h\in H_i\}
\]
be its generation level.  We claim that $h\ge\ell(h)+1$.  This is clear
for $\ell(h)=0$.  If $\ell(h)=r+1$, then
$h=\alpha x+\beta y$ for some $x,y\in H_r$, and at least one of $x,y$
has generation level $r$; otherwise $h$ would already belong to $H_r$.
By induction and the fact that every element of $H$ is at least $1$,
\[
   h\ge(r+1)+1=r+2.
\]
Consequently, for every $M>1$ the sets
\[
   H_i\cap[1,M]
\]
stabilize after finitely many steps.  Each $H_i$ is compact and has no
isolated points, while the uniqueness of the above representation implies
that it contains no nondegenerate interval: if an interval contained a point $y\in H_i$, one
could perturb $y$ slightly in the direction of a generator
$u\in H_0$ not occurring in its representation, with a small negative
coefficient from $\mathbb F_0$, contradicting uniqueness.  It follows that
$H$ is closed, perfect, and locally nowhere dense.

Write the complement of $H$ as the union of its maximal open intervals and
remove the left endpoint of each such interval.  If $L$ denotes the set of
these endpoints, put
\[
   H':=H\setminus L.
\]
The standard generalized-inverse construction gives a strictly increasing
bijection
\[
   f:I\to H'
\]
from a suitable unbounded interval $I$.  Moreover,
\[
   \alpha H'+\beta H'\subset H'.
\]
Indeed, $\alpha H+\beta H\subset H$, and if
$\alpha x+\beta y$ were the removed left endpoint of a complementary
interval, then moving $y$ slightly upward inside $H'$ would move the
weighted sum into that complementary interval, a contradiction.

Now fix $x_0\in I$ and choose $u\in H_0$ which occurs with positive
coefficient in the unique representation of $f(x_0)$.  Every element of
\[
   \alpha f(x_0)+\beta H'
\]
then has positive coefficient at $u$.  On the other hand, $H'$ contains
arbitrarily large elements whose coefficient at $u$ is zero.  Indeed, for
every $n$ the set
\[
   \lambda^n\bigl(H_0\setminus\{u\}\bigr)\subset H_n
\]
consists of uncountably many such elements, whereas only countably many
endpoints are removed; moreover $\lambda^n\to\infty$.

For each such element $w\in H'$, the point $z=f^{-1}(w)$ cannot belong
to the range of the section $F(x_0,\cdot)$.  Since $f^{-1}$ is increasing
and unbounded, these omitted values $z$ can be chosen arbitrarily large.
Since the section itself is unbounded, it cannot be continuous on any tail.  Hence it has infinitely many discontinuity points
on every tail.  The same argument applies to the horizontal sections.

\end{enumerate}
\end{rem}

The following result is an immediate consequence of
Theorem~\ref{thm:alphabeta-counterexample} and
Lemma~\ref{lem:weighted-quasisum}.
It shows that in Acz\'el's characterization of associative quasi-sums
(Theorem~\ref{Aczel_assoc}) the continuity assumption cannot be omitted.
In particular, there exist associative and strictly monotone operations of
quasi-sum type which are not continuous.

\begin{thm}\label{thm:assoc-counterexample}
There exists an associative, partially strictly increasing function
\[
   F:I^2\to I
\]
which is not continuous.
\end{thm}

\begin{proof}
Apply Theorem~\ref{thm:alphabeta-counterexample} with $\alpha=\beta=1$.
By Lemma~\ref{lem:weighted-quasisum}, the resulting operation
\[
   F(x,y)=f^{-1}\bigl(f(x)+f(y)\bigr)
\]
is associative (and bisymmetric) and partially strictly increasing in each
variable, while Theorem~\ref{thm:alphabeta-counterexample} gives its
discontinuity.
\end{proof}

\begin{rem}[One-point reflexivity in the associative case]\label{rem:onepoint}
In the example of Theorem~\ref{thm:assoc-counterexample} we can, in addition
to associativity, symmetry, and strict monotonicity, also obtain
\emph{one-point reflexivity}.  For $\alpha=\beta=1$, the expanding
construction of Lemma~\ref{lem:expanding-range}, applied with $\lambda=2$,
has domain $(1,\infty)$ and
range $H'\subset(0,\infty)$.  Reparametrize it by the increasing
homeomorphism
\[
   h:(0,\infty)\to(1,\infty),
   \qquad
   h(x)=x+1,
\]
and put $g:=f\circ h$.  Then
\[
   G(x,y):=g^{-1}\bigl(g(x)+g(y)\bigr),
   \qquad x,y>0,
\]
is associative, symmetric, bisymmetric, and partially strictly increasing,
and every section with a positive fixed argument is discontinuous.

Extend $G$ to $[0,\infty)^2$ by adjoining $0$ as a neutral element:
\[
   G(0,0):=0,\qquad
   G(0,x):=x,\qquad
   G(x,0):=x
   \qquad(x>0).
\]
This extension is still strictly increasing in each variable.  Indeed, for
$x,y>0$,
\[
   g(G(x,y))=g(x)+g(y)>\max(g(x),g(y)),
\]
and hence
\[
   G(x,y)>x=G(x,0),
   \qquad
   G(x,y)>y=G(0,y).
\]
Associativity and symmetry are preserved by adjoining a neutral element,
and together they imply bisymmetry.  The extended operation is reflexive
at $0$, whereas for $x>0$,
\[
   g(G(x,x))=2g(x)\neq g(x),
\]
so it is not reflexive at any other point.  Every section with a positive
fixed argument remains discontinuous, while the sections with fixed
argument $0$ are the identity map.

Consequently, reflexivity at one point does not force continuity.  In
Section~\ref{sec:two-point-refl} we will see that, in the symmetric case,
reflexivity at two distinct points already implies continuity on the
entire segment between them (Theorem~\ref{T:refl2p}).
\end{rem}

\section{Multivariate extensions}\label{sec:multi}

In this section we extend the construction from
Theorem~\ref{thm:alphabeta-counterexample} to $n$--ary operations.
We use the letter $F$ also in the multivariate setting, as no confusion
will arise.

\subsection{Basic notions for \texorpdfstring{$n$}{n}--ary operations}

Let $I$ be a real interval and let $F:I^n\to I$ be an $n$--ary operation.

\begin{itemize}

    \item \emph{Reflexivity.}
    The operation $F$ is reflexive if
    \[
       F(x,\dots,x)=x
       \qquad (x\in I).
    \]

    \item \emph{Symmetry.}
    The operation $F$ is symmetric if
    \[
      F(x_1,\dots,x_n)
      =F\bigl(x_{\sigma(1)},\dots,x_{\sigma(n)}\bigr)
    \]
    for every permutation $\sigma\in S_n$.

    \item \emph{Partial strict monotonicity.}
    The operation $F$ is partially strictly increasing if, for each $i$
    and each fixed choice of all variables except the $i$th one, the map
    \[
       t\mapsto
       F(x_1,\dots,x_{i-1},t,x_{i+1},\dots,x_n)
    \]
    is strictly increasing.

    \item \emph{Bisymmetry.}
    The $n$--ary bisymmetry equation is
    \begin{align*}
       &F\bigl(F(x_{1,1},\dots,x_{1,n}),
               F(x_{2,1},\dots,x_{2,n}),\dots,
               F(x_{n,1},\dots,x_{n,n})\bigr)\\
       &\qquad =
       F\bigl(F(x_{1,1},\dots,x_{n,1}),
               F(x_{1,2},\dots,x_{n,2}),\dots,
               F(x_{1,n},\dots,x_{n,n})\bigr),
    \end{align*}
    for all $x_{i,j}\in I$.

    \item \emph{Associativity.}
    Following the standard definition of an $n$--ary semigroup, $F$ is
    associative if, for every $i=1,\dots,n-1$,
    \begin{align*}
      &F\bigl(x_1,\dots,x_{i-1},
              F(x_i,\dots,x_{i+n-1}),
              x_{i+n},\dots,x_{2n-1}\bigr)\\
      &\qquad =
      F\bigl(x_1,\dots,x_i,
              F(x_{i+1},\dots,x_{i+n}),
              x_{i+n+1},\dots,x_{2n-1}\bigr)
    \end{align*}
    for all $(x_1,\dots,x_{2n-1})\in I^{2n-1}$.

\end{itemize}

\subsection{Weighted quasi--sums in \texorpdfstring{$n$}{} variables}

Fix coefficients $\alpha_1,\dots,\alpha_n>0$ and a strictly increasing
bijection $f:I\to H'$, where $H'\subset\R$ satisfies
\[
   \alpha_1H'+\cdots+\alpha_nH'\subset H'.
\]
Define
\begin{equation}\label{eq:multivar-def}
   F(x_1,\dots,x_n)
   :=f^{-1}\!\left(\sum_{i=1}^n\alpha_i f(x_i)\right).
\end{equation}

Partial strict monotonicity follows immediately from the positivity of the
coefficients and the strict monotonicity of $f$ and $f^{-1}$.
Bisymmetry follows directly from the additive form.  Indeed, for an
$n\times n$ matrix $(x_{i,j})$, applying $f$ to either side of the
$n$--ary bisymmetry equation gives
\[
   \sum_{i=1}^n\sum_{j=1}^n
      \alpha_i\alpha_j f(x_{i,j}).
\]
Similarly, $F$ is symmetric if and only if
\[
   \alpha_1=\cdots=\alpha_n,
\]
and it is reflexive if and only if
\[
   \sum_{i=1}^n\alpha_i=1.
\]

Finally, $F$ is associative if and only if
\[
   \alpha_1=\cdots=\alpha_n=1.
\]
Indeed, comparison of coefficients in the first associativity identity
gives $\alpha_1=\alpha_n=1$ and
\[
   \alpha_{k+1}=\alpha_2\alpha_k,
   \qquad k=1,\dots,n-1.
\]
Hence
\[
   \alpha_k=\alpha_2^{\,k-1},
   \qquad k=1,\dots,n,
\]
and therefore
\[
   1=\alpha_n=\alpha_2^{\,n-1}.
\]
Since the coefficients are positive, $\alpha_2=1$, and hence all weights
are equal to $1$.  Conversely, if all weights are $1$, then applying $f$
to any of the nested expressions gives the sum of the same $2n-1$ terms,
so $F$ is associative.

\subsection{Discontinuous bisymmetric \texorpdfstring{$n$}{n}--ary operations}

The $n$--ary analogues of Acz\'el-type results for continuous, strictly
monotone bisymmetric operations, including the reflexive and symmetric
setting, were studied by Maksa, Mokken, and M\"unnich~\cite{MaMoMu2000}.
The more general continuous, strictly monotone solutions of the general
bisymmetry equation were described by Maksa~\cite{Ma1999}.

We now show that, once reflexivity is omitted, discontinuous bisymmetric
examples exist also in the $n$--ary setting.  Acz\'el's theorem for
associative, strictly monotone, continuous $n$--ary operations was
generalized by Couceiro and Marichal~\cite{CouMar2012}; as a consequence
of our construction, Corollary~\ref{cor:assoc-n-ary} shows that the
continuity assumption cannot be omitted there either.

We first prove the $n$--ary analogue of
Theorem~\ref{thm:alphabeta-counterexample}.

\begin{thm}\label{thm:n-ary-counterexample}
Let $n\ge2$ and let $\alpha_1,\dots,\alpha_n>0$ satisfy
\[
   \lambda:=\sum_{i=1}^n\alpha_i\neq1.
\]
Then there exist a proper interval $I\subset\R$, a set $H'\subset\R$, and
a strictly increasing bijection $f:I\to H'$ such that
\[
   \alpha_1H'+\cdots+\alpha_nH'\subset H',
\]
and the operation
\[
   F(x_1,\dots,x_n)
   :=f^{-1}\!\left(\sum_{i=1}^n\alpha_i f(x_i)\right)
\]
is bisymmetric, partially strictly increasing, and non-reflexive, but not
continuous.  More precisely, for every choice of $(n-1)$ fixed variables,
the remaining one-variable section of $F$ is discontinuous.
\end{thm}

\begin{proof}
Assume first that $\lambda>1$, and take $I,H',f$ and $\gamma$ from
Lemma~\ref{lem:expanding-range}.
If $u_1,\dots,u_n\in H'$, then $u_i>1$ for every $i$, and hence
\[
   \sum_{i=1}^n\alpha_i u_i
      >\sum_{i=1}^n\alpha_i
      =\lambda.
\]
Therefore
\[
   \alpha_1H'+\cdots+\alpha_nH'
      \subset(\lambda,\infty)\subset H',
\]
so $F$ is well defined.

Fix all variables except the $j$th one and put
\[
   S:=\sum_{i\ne j}\alpha_i f(x_i).
\]
At $t=\gamma$ the argument of $f^{-1}$ is
\[
   A:=S+\alpha_j\lambda,
\]
whereas, as $t\uparrow\gamma$, it tends to
\[
   B:=S+\alpha_j\gamma.
\]
Both $A$ and $B$ are strictly larger than $\lambda$, and therefore lie in
the same upper component of $H'$, where
\[
   f^{-1}(s)=s-(\lambda-\gamma).
\]
Thus the jump of $f$ at $\gamma$ is carried through the affine branch of
$f^{-1}$, and
\[
\begin{split}
   &F(x_1,\dots,x_{j-1},\gamma,x_{j+1},\dots,x_n)\\
   &\qquad -
   \lim_{t\uparrow\gamma}
   F(x_1,\dots,x_{j-1},t,x_{j+1},\dots,x_n)
   =\alpha_j(\lambda-\gamma)>0.
\end{split}
\]
Hence every one-variable section is discontinuous at $\gamma$.

Now assume that $0<\lambda<1$, and take $I,H',f$ and $\gamma$ from
Lemma~\ref{lem:contracting-range}.  Since $0<u_i<1$ for every $u_i\in H'$,
\[
   0<
   \sum_{i=1}^n\alpha_i u_i
   <
   \sum_{i=1}^n\alpha_i
   =\lambda.
\]
Thus
\[
   \alpha_1H'+\cdots+\alpha_nH'
      \subset(0,\lambda)\subset H',
\]
and $F$ is again well defined.  On $(0,\lambda]$ the inverse $f^{-1}$ is
the identity.  Therefore, fixing all variables except the $j$th one,
\[
\begin{split}
   &\lim_{t\downarrow\lambda}
   F(x_1,\dots,x_{j-1},t,x_{j+1},\dots,x_n)\\
   &\qquad -
   F(x_1,\dots,x_{j-1},\lambda,x_{j+1},\dots,x_n)
   =\alpha_j(\gamma-\lambda)>0.
\end{split}
\]
Hence every one-variable section is discontinuous also in the contracting
case.

Finally, the weighted quasi--sum calculation in the preceding subsection
shows that $F$ is bisymmetric and partially strictly increasing.  Since
$\lambda\ne1$, it is non-reflexive.
\end{proof}

\begin{rem}\label{rem:nary-variants}
If
\[
   \alpha_1=\cdots=\alpha_n>0
   \qquad\text{and}\qquad
   \sum_{i=1}^n\alpha_i\ne1,
\]
then the operation in Theorem~\ref{thm:n-ary-counterexample} is symmetric.
Thus discontinuity persists even under full symmetry.

The stronger refinement from
Remark~\ref{rem:sections-and-variants}(iii) also has an $n$--ary analogue:
if
\[
   \alpha_i\ge1
   \qquad (i=1,\dots,n),
\]
then the same recursive Cantor-type construction can be used so that every
one-variable section has infinitely many discontinuity points on every
tail of the underlying interval.  The proof is completely analogous to
the binary case.
\end{rem}

\begin{cor}\label{cor:assoc-n-ary}
For every $n\ge2$ there exists an associative, symmetric, partially strictly
increasing $n$--ary operation
\[
   F:I^n\to I
\]
which is not continuous.
\end{cor}

\begin{proof}
Take
\[
   \alpha_1=\cdots=\alpha_n=1
\]
in Theorem~\ref{thm:n-ary-counterexample}.  Since
\[
   \lambda=n\ne1,
\]
the theorem gives a discontinuous, partially strictly increasing operation,
and the weighted quasi--sum calculation above shows that it is associative
and symmetric.
\end{proof}

\section{Local reflexivity at two points and continuity}\label{sec:two-point-refl}

Although reflexivity at a single point does not force continuity, as we have
seen in Remark~\ref{rem:onepoint}, a much stronger rigidity phenomenon
appears when reflexivity is available at two distinct points.
In this case, continuity is automatically restored on the whole interval
between these points.  The proof below adapts the argument of
\cite[Theorem~8]{BKSZ2021} to the present local setting; see also
\cite[Lemma~6]{BKSZ2023}.

\begin{thm}[Reflexivity at two points]\label{T:refl2p}
Let $I$ be an interval and let $F:I^2\to I$ be bisymmetric, symmetric,
and partially strictly increasing.
Assume that there exist $a,b\in I$, $a<b$, such that
\[
   F(a,a)=a,\qquad F(b,b)=b.
\]
Then the restriction of $F$ to $[a,b]^2$ is reflexive and continuous.
In particular, there exists a continuous strictly monotone function
$\phi:[a,b]\to\mathbb{R}$ such that
\[
   F(x,y)=\phi^{-1}\!\left(\frac{\phi(x)+\phi(y)}{2}\right)
   \qquad (x,y\in[a,b]).
\]
Thus $F$ is a quasi-arithmetic mean on $[a,b]$.
\end{thm}

\begin{proof}
First note that $F$ maps $[a,b]^2$ into $[a,b]$.  Indeed, if
$x,y\in[a,b]$, then monotonicity gives
\[
   a=F(a,a)\le F(x,y)\le F(b,b)=b.
\]
Thus throughout the proof we may work with the restriction
$F:[a,b]^2\to[a,b]$.

\medskip
\noindent\textbf{(1) Words generated by $a$ and $b$.}
Set $W_0:=\{a,b\}$ and define inductively
\[
   W_{n+1}:=W_n\cup\{F(x,y):x,y\in W_n\},
   \qquad n\ge0,
\]
and
\[
   W_\infty:=\bigcup_{n\ge0}W_n.
\]
By construction,
\[
   W_0\subset W_1\subset\cdots\subset W_\infty\subset[a,b].
\]

\medskip
\noindent\textbf{(2) Propagating reflexivity.}
We claim that every $w\in W_\infty$ is reflexive.  This follows by induction.
It is true for $W_0$ by assumption.  If it holds on $W_n$ and
$w=F(x,y)$ with $x,y\in W_n$, then bisymmetry gives
\[
   F(w,w)
   =F(F(x,y),F(x,y))
   =F\bigl(F(x,x),F(y,y)\bigr)
   =F(x,y)
   =w.
\]
Hence $F(w,w)=w$ for every $w\in W_\infty$.

\medskip
\noindent\textbf{(3) Dyadic parametrization.}
Let $\mathcal D$ be the set of dyadic rationals in $[0,1]$,
\[
   \mathcal D
   :=\left\{\frac{k}{2^n}:n\in\mathbb N,\ 0\le k\le2^n\right\}.
\]
As in Acz\'el's classical construction
\cite[pp.~287--290]{Aczel1989}, define
$f_0:\mathcal D\to W_\infty$ by
\[
   f_0(0):=a,\qquad f_0(1):=b,
\]
and require
\begin{equation}\label{eq:dyadic-identity}
   f_0\!\left(\frac{r+s}{2}\right)
   =F\bigl(f_0(r),f_0(s)\bigr)
   \qquad (r,s\in\mathcal D).
\end{equation}
The usual induction, using symmetry, bisymmetry, strict monotonicity, and
the reflexivity already established on $W_\infty$, shows that $f_0$ is
well defined and strictly increasing, that
\[
   f_0(\mathcal D)=W_\infty,
\]
and that \eqref{eq:dyadic-identity} holds for all $r,s\in\mathcal D$.

\medskip
\noindent\textbf{(4) Density of $W_\infty$.}
Let
\[
   C:=\overline{W_\infty}\subset[a,b].
\]
We first note that $C$ is uncountable.  For $x\in(0,1]$ define
\[
   \widehat f(x):=
   \sup\{f_0(d):d\in\mathcal D,\ d<x\},
   \qquad
   \widehat f(0):=a.
\]
Each $\widehat f(x)$ belongs to $C$.  If $0\le x<y\le1$, choose
$r,s\in\mathcal D$ with
\[
   x<r<s<y.
\]
Then
\[
   \widehat f(x)\le f_0(r)<f_0(s)\le\widehat f(y),
\]
so $\widehat f$ is strictly increasing.  Hence $C$ is uncountable.

Let $L$ be the set of two-sided accumulation points of $C$.  Since $C$ is
closed, every point of $C$ which is not a two-sided accumulation point is
an endpoint of a component of $\mathbb R\setminus C$ (apart from the two
endpoints of $C$).  There are only countably many such components.
Consequently, $L$ is uncountable.

Suppose, for a contradiction, that $W_\infty$ is not dense in $[a,b]$.
Then the monotone map $f_0$ has a jump at some cut $z\in(0,1)$: there are
dyadic sequences
\[
   d_n\nearrow z,\qquad D_n\searrow z
\]
such that
\[
   X:=\lim_{n\to\infty}f_0(d_n)
   <\lim_{n\to\infty}f_0(D_n)=:Y.
\]
We may choose the sequences so that
\[
   f_0(d_n)<X<Y<f_0(D_m)
   \qquad(n,m\in\mathbb N).
\]

Take two distinct points $s,t\in L$ with $s<t$.  Since $C$ is the closure
of $W_\infty$, there exist $r_s,r_t\in\mathcal D$ such that
\[
   s<f_0(r_s)<f_0(r_t)<t.
\]
For sufficiently large $n,m$,
\[
   \frac{D_m+r_s}{2}
   <
   \frac{d_n+r_t}{2}.
\]
Using \eqref{eq:dyadic-identity} and the strict monotonicity of $f_0$, we
obtain
\begin{align*}
   F\bigl(f_0(d_n),f_0(r_s)\bigr)
   &<F\bigl(f_0(D_m),f_0(r_s)\bigr)\\
   &<F\bigl(f_0(d_n),f_0(r_t)\bigr)
   <F\bigl(f_0(D_m),f_0(r_t)\bigr).
\end{align*}
Since
\[
   f_0(d_n)<X<Y<f_0(D_m)
\]
and
\[
   s<f_0(r_s)<f_0(r_t)<t,
\]
strict monotonicity of $F$, together with the middle inequality above,
gives
\[
   F(Y,s)
   <F\bigl(f_0(D_m),f_0(r_s)\bigr)
   <F\bigl(f_0(d_n),f_0(r_t)\bigr)
   <F(X,t).
\]
Therefore
\[
   (F(X,s),F(Y,s))
   \cap
   (F(X,t),F(Y,t))
   =\emptyset.
\]
Thus the nondegenerate open intervals
\[
   (F(X,s),F(Y,s)),
   \qquad s\in L,
\]
are pairwise disjoint.  They all lie in $[a,b]$, which is impossible
because $L$ is uncountable.  Hence $W_\infty$ is dense in $[a,b]$.

\medskip
\noindent\textbf{(5) Extension and quasi-arithmetic form.}
Since $f_0:\mathcal D\to[a,b]$ is strictly increasing and
$f_0(\mathcal D)=W_\infty$ is dense in $[a,b]$, it extends uniquely to a
strictly increasing homeomorphism
\[
   f:[0,1]\to[a,b].
\]
Put $\phi:=f^{-1}$.  For $d_1,d_2\in\mathcal D$,
\[
   F\bigl(f(d_1),f(d_2)\bigr)
   =f\!\left(\frac{d_1+d_2}{2}\right)
   =\phi^{-1}\!\left(
       \frac{\phi(f(d_1))+\phi(f(d_2))}{2}
     \right).
\]
To pass from the dense set $W_\infty^2$ to all of $[a,b]^2$, approximate
$x$ and $y$ from below and above by points of $W_\infty$.  Monotonicity of
$F$ sandwiches $F(x,y)$ between the corresponding values on
$W_\infty^2$, and the right-hand side above is continuous.  Therefore
\[
   F(x,y)
   =\phi^{-1}\!\left(\frac{\phi(x)+\phi(y)}{2}\right)
   \qquad(x,y\in[a,b]).
\]
The restriction of $F$ to $[a,b]^2$ is therefore continuous and reflexive.
\end{proof}

\section{Concluding remarks and open problems}\label{sec:conclusion}

The main result of this paper shows that \emph{bisymmetry together with
strict monotonicity does not imply continuity}.  We constructed large
classes of bisymmetric, partially strictly increasing operations for which
every one-variable section is discontinuous, even in the symmetric case.
In particular, outside the reflexive regime, continuity cannot be dropped
from Acz\'el's quasi--sum theorem (Theorem~\ref{Aczel3}(2)).

This naturally leads to the following general problems.

\begin{que}
How can one characterize bisymmetric, strictly increasing operations
\[
   F:I^n\to I,\qquad n\ge2,
\]
without continuity assumptions?
\end{que}

\begin{que}
How can one characterize associative, strictly increasing operations
\[
   F:I^n\to I,\qquad n\ge2,
\]
without continuity assumptions?
\end{que}

On the other hand, Theorem~\ref{T:refl2p} shows that in the symmetric
binary case \emph{two-point reflexivity} already forces continuity and a
quasi-arithmetic representation on the segment spanned by the two reflexive
points.  Together with \cite{BKSZ2021,BKSZ2023}, this highlights the strong
regularizing effect of reflexivity when combined with symmetry.

\medskip
A central open problem in the binary theory is the following.

\begin{que}[Global problem on continuity]\label{que:global}
Is every \emph{reflexive}, bisymmetric, partially strictly increasing binary
operation $F:I^2\to I$ necessarily continuous?
Equivalently, does reflexivity together with bisymmetry and strict
monotonicity already force $F$ to be weighted quasi-arithmetic?
\end{que}

This is, in our view, one of the main questions left open by
\cite{BKSZ2021,BKSZ2023} and the present work.  Without reflexivity,
discontinuous examples exist by
Theorem~\ref{thm:alphabeta-counterexample}.  In the reflexive case, symmetry
is not available a priori, so the symmetric regularity theorem cannot be
applied directly.

At the time when the first version of the present manuscript was submitted,
the corresponding continuity problem for reflexive, symmetric, bisymmetric,
partially strictly increasing $n$--ary operations was known only for the
special cases $n=2^k$.  The question raised here motivated subsequent work
\cite{KissShulman2026}, in which the continuity result was extended to
every $n\ge2$.  The general non-symmetric reflexive problem remains open.
\medskip
Our two-point reflexivity theorem also raises questions about how local
reflexivity can propagate.

\begin{que}[Extension of Theorem~\ref{T:refl2p}]\label{que:extend}
Let $F:I^2\to I$ be bisymmetric, symmetric, and partially strictly
increasing.  If
\[
   F(a,a)=a,\qquad F(b,b)=b
\]
for some $a<b$, must $F$ be continuous on the entire interval $I$, rather
than only on $[a,b]$?
\end{que}

A related question is whether reflexivity on a nondegenerate subinterval
must propagate to the whole domain.

\begin{que}[From local reflexivity]\label{que:local-reflex}
Let $F:I^2\to I$ be bisymmetric, symmetric, and partially strictly
increasing.  If $F(x,x)=x$ for every $x$ in some nondegenerate subinterval
$J\subset I$, must $F$ be reflexive on all of $I$?
\end{que}

An affirmative answer would imply, by
Theorem~\ref{T:bisymmetryimpliescontinuity}, that $F$ is continuous and
quasi-arithmetic on the whole interval $I$.

Finally, the phenomenon of \emph{one-point reflexivity} remains poorly
understood.  Our associative example
($\alpha=\beta=1$, Remark~\ref{rem:onepoint}) shows that a single reflexive
point can occur at a boundary point introduced as a neutral element while
the operation remains discontinuous in the interior.

\begin{que}[One-point reflexivity]\label{que:onepoint}
In the associative or in the general bisymmetric setting, can a
discontinuous partially strictly increasing operation $F:I^2\to I$ be
reflexive at an interior point $c\in I$?
Does one-point interior reflexivity force any local continuity?
\end{que}

Further progress on these problems appears to require a better
understanding of the interaction between the algebraic properties of
bisymmetry, associativity, symmetry, and reflexivity and the order-theoretic
regularity of monotone operations.

\section*{Acknowledgements}

The author is grateful to Mikl\'os Laczkovich for inspiring conversations. The author also thanks the anonymous referee for the careful reading of the
manuscript and for pointing out a gap in the original proof, which led to a
substantial simplification and improvement of the main construction.

\section*{Statements and Declarations}

\textbf{Funding}

\noindent
This work was supported by the Hungarian National Research, Development and Innovation Office (NKFIH) Grants K146922, STARTING 150576, and FK142993.

\noindent
\textbf{Competing Interests}

\noindent
The author declares that he has no competing interests.

\noindent
\textbf{Author Contributions}

\noindent
The author is solely responsible for the conception of the work, the mathematical results, and the writing of the manuscript.

\noindent
\textbf{Ethics Approval}

\noindent
Not applicable.

\noindent
\textbf{Consent to Participate}

\noindent
Not applicable.

\noindent
\textbf{Consent for Publication}

\noindent
Not applicable.

\noindent
\textbf{Data Availability}

\noindent
No datasets were generated or analysed during the current study.

\noindent
\textbf{Code Availability}

\noindent
Not applicable.

\noindent
		{\sc Gergely Kiss:}\\
        Corvinus University of Budapest, Department of Mathematics \\
		Fővám tér 13-15, Budapest 1093, Hungary,\\
        and\\
		HUN-REN Alfr\'ed R\'enyi Mathematical Institute\\
		Re\'altanoda utca 13-15, H-1053, Budapest, Hungary\\
		E-mail: {\tt kiss.gergely@renyi.hu}

\end{document}